\documentclass[11pt,a4paper]{article}

\usepackage{caption2}
\usepackage{amssymb}
\usepackage{amsfonts}
\usepackage{mathrsfs}
\usepackage{empheq}
\usepackage{amsmath,amsthm}
\usepackage{graphicx}
\usepackage{color}
\usepackage{indentfirst}
\usepackage{hyperref}
\usepackage{float}
\usepackage{cases}
\usepackage[titletoc]{appendix}

\allowdisplaybreaks%

\newtheorem{thm}{Theorem}[section]

\newtheorem{lem}{Lemma}[section]

\theoremstyle{definition}
\newtheorem{defn}{Definition}[section]
\newtheorem{rem}{Remark}[section]
\numberwithin{equation}{section}

\def\i1n{i=1,\cdots,n}
\def\j1n{j=1,\cdots,n}
\def\ij1n{i,j=1,\cdots,n}

\def\R{\mathbb R}

\newcommand{\be}{\begin{equation}}
	\newcommand{\ee}{\end{equation}}
\newcommand{\beq}{\begin{equation*}}
	\newcommand{\eeq}{\end{equation*}}

\begin{document}

	\title
	{\bf Existence of solutions to a generalized self-dual Chern-Simons equation on graphs}
	\author{Yingshu L\"{u} \footnote{Email address: yingshulv@fudan.edu.cn. }
		\hspace{.2in} Peirong Zhong \footnote{Email address: 19210180058@fudan.edu.cn.}
		\\ \\}
	\date{}
	\maketitle
	\begin{center}
School of Mathematical Sciences, Fudan University, Shanghai, China
\end{center}
\date{}
\maketitle
\begin{abstract}
Let $ G=(V,E) $ be a connected finite graph and $ \Delta $ the usual graph Laplacian. In this paper, we consider a generalized self-dual Chern-Simons equation on the graph $G$
\begin{eqnarray}\label{one1}
\Delta{u}=-\lambda{e^{F(u)}[e^{F(u)}-1]^2}+4\pi\sum_{i=1}^{M}{\delta_{p_{j}}},
\end{eqnarray}
where
\begin{align} \nonumber
	F(u)=
	\begin{cases}
		\widetilde{F}(u), \   \quad u\leq0,\\
		0, \quad \quad \quad u>0,
	\end{cases}
\end{align}
$ \widetilde{F}(u) $ satisfies
$ u=1+{\widetilde {F}(u)}-e^{\widetilde {F}(u)} $, $ \lambda>0 $,
$M$ is any fixed positive integer,
$ \delta_{p_{j}} $ is the Dirac delta mass at the vertex $p_j$,
and $p_1$, $p_2$, $\cdots$, $p_M$ are arbitrarily chosen distinct vertices on the graph. We first prove that there is a critical value ${\lambda}_c$ such that if $\lambda \geq{\lambda}_c$, then the generalized self-dual Chern-Simons equation has a solution $u_{\lambda}$. Applying the existence result, we develop a new method to construct a solution of the equation \eqref{one1} which is monotonic with respect to $\lambda$ when $\lambda \geq{\lambda}_c$. Then we establish that there exist at least two solutions of the equation via the variational method for $\lambda>{\lambda}_c$. Furthermore, we give a fine estimate of the monotone solution which can be applied to other related problems.
\end{abstract}
	
	\section{Introduction}
	
In this paper, we study the following generalized self-dual Chern-Simons equation on a connected finite graph $G$
	\begin{equation}\label{one}
		\Delta{u}=-\lambda{e^{F(u)}[e^{F(u)}-1]^2}+4\pi\sum_{i=1}^{M}{\delta_{p_{j}}},
	\end{equation}
where $\Delta$ is the graph Laplacian,
    \begin{align}
		\label{fu}
		F(u)=
		\begin{cases}
			\widetilde{F}(u), \   \quad u\leq0,\\
			0, \quad \quad \quad u>0,
		\end{cases}
	\end{align}
	$ \widetilde{F}(u) $ satisfies
	$ u=1+{\widetilde{F}(u)}-e^{\widetilde{F}(u)} $, $p_1,p_2,...,p_M$ are distinct points in the graph $G$ and $ \delta_{p_{j}} $ satisfies
	\begin{align} \nonumber
		\delta_{p_{j}}=
		\begin{cases}
			1, \   \quad \mbox{at $ p_j $}\\
			0, \quad \mbox{otherwise}.
		\end{cases}
	\end{align}
Note that if $F(u)=u$, then the solutions of the generalized self-dual Chern-Simons model is referred to as vortices. We remark that the notion of vortices play important roles in many aspects of sciences including superconductivity \cite{JT}, optics \cite{BE}, quantum Hall effects \cite{S}, for which one can read.

The study of vortices in (2+1)-dimensional Chern-Simons gauge theory has attracted much attention recently. One of the important features of these vortices, which differs from Nielsen-Olesen vortices \cite{NO}, is that they are magnetically and elecrically charged. In the Chern-Simons model, the Yang-Mills (or Maxwell) term does not appear in the action Lagrangian density and only the Chern-Simons term governs eletromagnetism. Under the condition that the Higgs potential takes a sextic form, the static equations of motion can be deduced by reducing a system of second-order differential equations to a self-dual system of first-order equations, and then the topological multivortices \cite{SY},\cite{W}, non-topological multivortices\cite{CHMY}-\cite{CFL},\cite{SY2}, and doubly periodic vortices \cite{H} can all be studied rigorously in mathematical methods.

A generalized self-dual Chern-Simons model was later proposed by \cite{HKP} and now plays an important role in various areas of physics, many researchers did a lot of significant work on the existence of non-topological vortices and topological vortices in this Chern-Simons model \cite{CI1},\cite{TY},\cite{Y}. However, the existence of double periodic vortices in this Chern-Simons model had been an open problem, until recently Han solved this problem in \cite{H}. He reduced the generalized self-dual Chern-Simons equation to a quasilinear elliptic equation by appropriate transformations, and established the existence of double periodic vortices of the Chern-Simons model by the methods of subsolutions and supersolutions. In this paper, we investigate the existence of the solutions to the generalized Chern-Simons equation on a connected finite graph.

In recent years, the research on the elliptic equations on graphs has attracted attention increasingly from scientists. Grigor'yan, Lin and Yang \cite{GLY} considered the Kazdan-Warner equation on a finite graph, where the Kazdan-Warner equation was initially studied on a manifold \cite{KW}. In \cite{GLY}, they gave the solvability of the following equation depending on the sign of $c$
$$\Delta u=c-he^u,$$
where $c$ is a constant and $h: V \to \mathbb{R}$ is a function. For more results of solvability of the Kazdan-Warner equation, we refer readers to \cite{GLY3}-\cite{GWK},\cite{KS}. The Kazdan-Warner equation is closely related to the mean field equation investigated originally in the prescribed curvature problem in geometry. Huang, Lin and Yau \cite{HLYS} proved the existence of solution to the following two mean field equations
\begin{eqnarray}\label{m}
 \Delta u + e^u = \rho \delta_0,
\end{eqnarray}
and
  $$ \Delta u =\lambda e^u (e^u-1) + 4\pi \sum_{j=1}^{M} \delta_{p_j} $$
on an arbitrary connected finite graph, where $\rho>0$ and $\lambda>0$ are constants.

 From these results, we study the generalized Chern-Simons equation \eqref{one} on a finite graph $G$. Inspired by the idea from \cite{TG}, we first show that there is a critical value $\lambda_c$ depending on the graph $G$ such that if $\lambda>\lambda_c$, the equation \eqref{one} has a solution via the variation method. Moreover, the existence of the solutions to the equation \eqref{one} also holds when $\lambda=\lambda_c$ by the properties of finite graphs. Applying the existence result for $\lambda\geq {\lambda}_{c}$, we put forward a new idea to construct a solution of the equation \eqref{one} which is monotonic with respect to $\lambda$. With the aid of the existence results and Mountain Pass theorem, we show that there exist at least two solutions of the equation \eqref{one}. The crucial point is that we apply the properties of the equation on the finite graph to prove the Palais-Smale condition used in Mountain Pass theorem which is different from the equations on Euclidean spaces. Furthermore, under the monotonicity property of the solutions, we give a fine estimate of the solutions for almost everywhere $\lambda>{\lambda}_c$.

The organization of the paper is as follows. In Section 2, we introduce the notations and preliminaries of the paper and then state our main results. In Section 3, we present the existence of the solution to the generalized self-dual Chern-Simons equation. A fine result about the solutions is established in Section 4.

\section{Settings and Main Results}
	Let $ G=(V,E) $ be a connected finite graph. For any edge $ xy\in E $, we assume that the symmetric weights $ \omega_{xy}=\omega_{yx} $ satisfying $\omega_{xy}>0$. Let $ \mu:V\to {\R^+} $ be a finite measure. For any function $ u:V\to \R $, the Laplace operator acting on $ u $ is defined by

    \begin{equation}
    	\label{two}
    	\Delta{u}(x)=\frac{1}{\mu(x)}\sum_{y\sim{x}}\omega_{yx}[u(y)-u(x)],
    \end{equation}
    where $y\sim x$ means ${xy}\in E$. The associated gradient form stands for
    \begin{eqnarray}
    \Gamma{(u,v)}(x)=\frac{1}{2\mu (x)} \sum_{y\sim x}\omega_{xy}(u(y)-u(x))(v(y)-v(x))
    \end{eqnarray}
for any $u, v:V\to \R$.

For any function $f: V \to \mathbb{R}$, an integral of $f$ over $V$ is defined by
\begin{eqnarray}
\int_V f d\mu= \sum_{x\in V} \mu(x)f(x),
\end{eqnarray}
and then
    \begin{equation}
    	\label{gra}
    	\frac{1}{2} \int_V {|\nabla u|}^2 d\mu
    	:=\frac{1}{2}\sum_{x \in V} \sum_{y\sim{x}}
    	\omega_{xy} {[u(y)-u(x)]}^2.
    \end{equation}

Define a Sobolev space and a norm by
\begin{eqnarray*}
W^{1,2}(V)=\{u: V \to \mathbb{R}| \int_V (|\nabla u|^2+u^2)d\mu <+\infty\},
\end{eqnarray*}
and
\begin{eqnarray*}
\|u\|_{W^{1,2}(V)}=\left(\int_V(|\nabla u|^2+u^2)d\mu\right)^{\frac{1}{2}}
\end{eqnarray*}
respectively. Since $G$ is a finite graph, the Sobolev space $W^{1,2}(V)$ is the set of all functions on $V$, a finite dimensional linear space. Then the following Sobolev embedding was introduced in [\cite{GLY}, Lemma 5].
\begin{lem}\label{le1}
Let $G=(V, E)$ be a finite graph. The Sobolev space $W^{1,2}(V)$ is pre-compact. Namely, if $\{u_j\}$ is bounded in $W^{1,2}(V)$, then there exists some $u\in W^{1,2}(V)$ such that up to a subsequence, $u_j\to u$ in $W^{1,2}(V)$.
\end{lem}

Let $(X, \|\cdot\|)$ be a Banach space, $J: X \to \mathbb{R}$ be a functional. In order to prove the existence of the solutions of the equation (\ref{one}), we first need the following Palais-Smale condition.

\begin{defn}
A functional $J\in C^1(X, \mathbb{R})$ satisfies the Palais-Smale condition if each sequence $\{u_k\} \subset X$ such that
\begin{flalign}\nonumber
		 &(1)\  J(u_k)  \to c \ \mbox{for some constant} \ c, \  as \  k \to + \infty, &&\\ \nonumber
		 &(2)\  \Vert J'(u_k) \Vert \to 0 \ \mbox{in} \ X, \  as \  k \to + \infty&& \nonumber
	\end{flalign}
is precompact in $X$.
\end{defn}

With the help of the Palais-Smale condition, Ambrosetti and Rabinowitz established the following Mountain Pass theorem \cite{A}.

\begin{lem} \label{le2}
Let $(X, \|\cdot\|)$ be a Banach space, $J\in C^1(X, \mathbb{R})$, $e\in X$ and $r>0$  such that $\|e\|>r$ and
\begin{eqnarray*}
b:=\inf_{\|u\|=r} J(u)>J(0)\geq J(e).
\end{eqnarray*}		
If $J$ satisfies the Palais-Smale condition with $c:=\inf_{\gamma\in \Gamma}\max_{t\in[0,1]}J(\gamma{(t)})$, where
		\begin{equation*}
			\Gamma := \{ \gamma \in C([0,1], X) | \  \gamma(0)=0, \gamma(1)=e \},
		\end{equation*}
		then $c$ is a critical value of $J$.
	\end{lem}

Our first main result in this paper can be stated as follows.
    \begin{thm}
    	\label{thma}
    	There is a critical value $ \lambda_c $ depending on $ G $ satisfying
    	\begin{equation}
    		\lambda_c\ge \frac{27\pi M}{|V|}
    	\end{equation}
    such that\\
    (i) If $ \lambda \ge \lambda_c $, the equation (\ref{one}) has a solution $u_{\lambda}$ on $ G $, and if $ \lambda< \lambda_c $, the equation (\ref{one}) has no solution.\\
    (ii) If $ \lambda \ge \lambda_c $, there exists a solution $\tilde{u}_{\lambda}$ of the equation (\ref{one}) on $ G $ satisfying $\tilde{u}_{\lambda_1}\geq \tilde{u}_{\lambda_2}$ when $\lambda_1>\lambda_2$.\\
    (iii) If $ \lambda > \lambda_c $, the equation (\ref{one}) admits at least two solutions on $ G $.
    \end{thm}

Let $u_0$ be a solution of the Poisson equation
\begin{eqnarray}\label{k1}
\Delta u_0=-\frac{4\pi M}{|V|}+4\pi \sum_{j=1}^{M}\delta_{p_j}.
\end{eqnarray}
As is well known, the solution of (\ref{k1}) exists if the integral of the right hand side is equal to 0. Inserting $u=u_0+v$ into the equation \eqref{one} yields
\begin{eqnarray}\label{one2}
\Delta v= -\lambda{e^{F(u_0+v)}[e^{F(u_0+v)}-1]^2}+\frac{4\pi M}{|V|}.
\end{eqnarray}
  By Theorem \ref{thma}, we investigate the existence result of the equation \eqref{one} and show that $\tilde{v}_{\lambda}(=\tilde{u}_{\lambda}-u_0) $ is a local minimum of the functional related to the equation \eqref{one2} when $ \lambda > \lambda_c $. For the solution $\tilde{v}_{\lambda}$, we give a further result as follows.
    \begin{thm}
    	\label{thmc}
The solution $\tilde{v}_{\lambda}$ is a strict local minimum of the functional related to the equation \eqref{one2} for almost everywhere $ \lambda > \lambda_c $.
\end{thm}
\begin{rem}
(i) In \cite{HS}, Hou and Sun gave the same existence result of the equation \eqref{one} as $(i)$ in Theorem \ref{thma} by the method of the subsolutions and supersolutions for $ \lambda > \lambda_c $  and further proved that the solution $\hat{u}_{\lambda}$ is increasing respect to $\lambda$. In this paper, we prove the existence result of \eqref{one} by another method, i.e. the variational method. Since our method is variational, we cannot get the monotonicity of the solution directly. This implies that the solution $\hat{u}_{\lambda}$ (obtained in \cite{HS}) may be different from the solution $u_{\lambda}$ of \eqref{one}.

(ii) Hou and Sun \cite{HS} established the existence result of \eqref{one} at $\lambda=\lambda_c$ by the integral estimation method. In our paper, we apply properties of the equation on the finite graph and the contradiction argument.

(iii) We give not only the existence of the single solution, but also the existence of multiple solutions. To achieve this, we first put forward a new method to construct a solution $\tilde{u}_{\lambda}$ of the equation \eqref{one} which is monotonic with respect to $\lambda$. This implies that $\tilde{v}_{\lambda}$ is monotonic with respect to $\lambda$. Then we establish that $\tilde{v}_{\lambda}$ is a local minimum of the functional related to the equation \eqref{one2} when $ \lambda > \lambda_c $ by the monotonicity of $\tilde{v}_{\lambda}$. With the help of this result and the variational method, we show that the equation (\ref{one}) admits at least two solutions on $ G $ when $ \lambda > \lambda_c $.
\end{rem}

\section{Existence of solutions for the Chern-Simons equations}
    This section devotes to the proof of Theorem \ref{thma}. We establish the existence result of the solutions to the equation \eqref{one} by the variation methods. To achieve this purpose, we first prove some basic lemmas.

\begin{defn}
	We say that	$ \underline{u} $ is a {\bf sub-solution} of (\ref{one}) if
	\begin{equation}
		\label{four}
		\Delta{\underline{u}}\ge -\lambda{e^{F(\underline{u})}[e^{F(\underline{u})}-1]^2}
		+4\pi\sum_{j=1}^{M}{\delta_{p_{j}}}
	\end{equation}
for any $ x \in V $.
\end{defn}

\begin{lem}
	\label{lemb}
	If $ \underline{u} $ is a sub-solution to (\ref{one}), then $ \underline{u} $ is non-positive on $ V $.
\end{lem}

\begin{proof}	
	Let
	\begin{equation}
		\label{set}
		\Omega_1=\{x\in V|{\underline{u}}(x)> 0\}, \quad\
		\Omega_2=\{x \in V|{\underline{u}}(x)\le 0\} .
	\end{equation}

If $ \Omega_1 $ is empty, then the conclusion holds. If $ \Omega_1 $ is non-empty, we first claim that $\Omega_1\neq V$. Suppose that the claim does not hold, then
\begin{equation}\label{w}
		\Delta{\underline{u}}\ge 4\pi\sum_{j=1}^{M}{\delta_{p_{j}}} \ \quad \mbox{in} \quad {\Omega}_1.
	\end{equation}
Summing the two sides of the above equation, one can obtain that the left hand side of (\ref{w}) is equal to zero and the right side of (\ref{w}) is equal to $4\pi M$. This leads to a contradiction.

	Since
	\begin{equation*}
		\Delta{\underline{u}}\ge -\lambda{e^{F(\underline{u})}[e^{F(\underline{u})}-1]^2}
		+4\pi\sum_{j=1}^{M}{\delta_{p_{j}}} \quad \mbox{in} \quad {\Omega}_1,
	\end{equation*}
	and from (\ref{fu}), one has
	\begin{equation*}
		\Delta {\underline{u}} \ge 4\pi\sum_{j=1}^{M}{\delta_{p_{j}}}\geq 0 \quad \mbox{in} \quad {\Omega}_1.
	\end{equation*}
	Therefore,  $ \underline{u} $ is a sub-harmonic function in $\Omega_1$. Since $\Omega_1$ and $\Omega_2$ are non-empty, it follows from the maximum principle on graphs that
		\begin{equation*}
			\max\limits_{{\Omega}_1} \underline{u}
			\leq 0.
		\end{equation*}
	
	    This implies that
	    \begin{equation*}
	         \underline{u}(x) \le 0 \quad \mbox{in} \quad {\Omega}_1,
        \end{equation*}
        which contradicts with the definition of $\Omega_1$. Therefore, $ \underline{u} $ is non-positive on $ V $.
	\end{proof}

\begin{rem}\label{rm}
If $ u$ is a solution to (\ref{one}), then $ u $ is non-positive on $ V $.
\end{rem}

   According to (\ref{one}) and Remark \ref{rm}, one can easily see that
    \begin{equation}
    	\label{six}
    	\Delta u \ge -\frac{4}{27} \lambda +4 \pi \sum_{j=1}^{M} \delta_{p_j}.
    \end{equation}
    Summing the two sides of the above equation yields
    \begin{equation}
    	\label{sev}
    	\lambda \ge \frac{27 \pi M}{|V|}.
    \end{equation}
    This is a necessary condition for the existence of solutions to (\ref{one}).

In order to prove the existence of the solutions to the equation \eqref{one}, with the aid of the decomposition $u=u_0+v$ and the equation \eqref{one2}, we only need to consider the existence result of the equation \eqref{one2}.
\begin{lem}\label{im}
If $\underline{v}$ is a subsolution of the equation \eqref{one2}, then there exists a solution $v^*$ of the equation \eqref{one2}.
\end{lem}
\begin{proof}
Define the functional $ J:W^{1,2}(V) \to \R $ by
\begin{equation}
	\label{k2}
	J(v)= \frac{1}{2} \int_V {|\nabla v |}^2 +\frac{\lambda}{4} \int_V [e^{F(u_0+v)}-1]^4 +\frac{4\pi M}{|V|}\int_V v,
\end{equation}
and let the functional $J$ among all functions $v$ belonging to the set
\begin{eqnarray}
\mathcal{A}:= \{v\in W^{1,2}(V) | v\geq \underline{v} \ a.e. \ \mbox{in} \ V\},
\end{eqnarray}

Denote
\begin{eqnarray*}
J_0=\inf_{\mathcal{A}}J(v).
\end{eqnarray*}
Then there exists a sequence $\{v_n\}\subset \mathcal{A}$ such that $J(v_n)\to J_0$. We claim that $ \{v_{n}  \} $ is bounded in $W^{1,2}(V)$. In fact, we only need to prove that $ \{v_{n} (x) \} $ is bounded for any $x\in V$. Suppose not, then there exists a point $ x^k \in V$ and a subsequence $ \{v_{n_k}   \}\subset \{v_{n}\} $ such that
\begin{eqnarray}
v_{n_k}(x^k) \rightarrow +\infty, \ \ \mbox{as} \ k\to +\infty.
\end{eqnarray}
It follows from (\ref{k2}) that
\begin{eqnarray}
J(v_{n_k}) &=& \frac{1}{2} \int_V {|\nabla v_{n_k} |}^2 +\frac{\lambda}{4} \int_V [e^{F(u_0+v_{n_k})}-1]^4 +\frac{4\pi M}{|V|}\int_V v_{n_k} \\ \nonumber
&\geq& \frac{4\pi M}{|V|}\int_V v_{n_k} \\ \nonumber
&=& \frac{4\pi M}{|V|}\int_{V\backslash \{x^k\}} v_{n_k}+ \frac{4\pi M}{|V|}\mu(x^k)v_{n_k}(x^k) \rightarrow  +\infty, \ \ \mbox{as} \ k\to +\infty.
\end{eqnarray}
This is impossible. Thus, $ \{v_{n}(x)  \} $ is bounded for any $x\in V$, and therefore $ \{v_{n}  \} $ is bounded in $W^{1,2}(V)$. Applying the Sobolev embedding theorem(Lemma \ref{le1}), there exists a subsequence $ \{v_{n_l}   \}\subset \{v_{n}\} $ and $v^*\in W^{1,2}(V)$ such that $v_{n_l}    \to v^*$ as $l \to +\infty$. One can easily see that $v^*\in \mathcal{A}$. Hence, one has
\begin{eqnarray*}
J_0=J(v^*).
\end{eqnarray*}

In the following, we prove that $v^*$ is a critical point for $J$ in $W^{1,2}(V)$.

For any $\psi\in W^{1,2}(V)$, $\tau>0$, let
\begin{eqnarray}
v^\tau=v^*+\tau \psi + {\psi}^{\tau},
\end{eqnarray}
with ${\psi}^{\tau}=\{v^*+\tau \psi-\underline{v}\}_{-}\geq 0 \ \ a.e. \ \mbox{in} \ V$. Then, one has
\begin{eqnarray}
v^\tau \geq \underline{v} \ \ \ \ a.e. \ \mbox{in} \ V.
\end{eqnarray}
This implies that $v^\tau \in \mathcal{A}$. Hence, we have
\begin{eqnarray*}
0 &\leq& \frac{J(v^{\tau})-J(v^*)}{\tau}\\ \nonumber
&=& \frac{1}{2\tau} \int_V ({|\nabla v^{\tau} |}^2-|\nabla v^* |^2 )+\frac{\lambda}{4\tau} \int_V \{[e^{F(u_0+v^{\tau})}-1]^4-[e^{F(u_0+v^{*})}-1]^4\} +\frac{4\pi M}{|V|\tau}\int_V (v^{\tau}- v^*) \\ \nonumber
&=&  \frac{1}{\tau} \int_V (\nabla v^{*}\cdot \nabla (\tau \psi+{\psi}^{\tau}) )+ \frac{1}{2\tau} \int_V{|\nabla (\tau \psi+{\psi}^{\tau}) |}^2 -\lambda\int_V{e^{F(u_0+v^{*})}[e^{F(u_0+v^{*})}-1]^2}\psi\\ \nonumber
&& \ \ +\frac{\lambda}{4\tau} \int_V \left[\{[e^{F(u_0+v^{\tau})}-1]^4-[e^{F(u_0+v^{*})}-1]^4\}+4\tau {e^{F(u_0+v^{*})}[e^{F(u_0+v^{*})}-1]^2}\psi \right]\\ \nonumber
&& \ \ \ +\frac{4\pi M}{|V|}\int_V \psi +\frac{4\pi M}{|V|\tau}\int_V {\psi}^{\tau}.  \\ \nonumber
\end{eqnarray*}
Thus,
\begin{eqnarray*}
&& \int_V \nabla v^{*}\cdot \nabla  \psi -\lambda\int_V {e^{F(u_0+v^{*})}[e^{F(u_0+v^{*})}-1]^2}\psi+\frac{4\pi M}{|V|}\int_V \psi\\ \nonumber
&\geq&-\frac{1}{\tau}\int_V \nabla v^{*}\cdot \nabla  {\psi}^{\tau}-\frac{1}{2\tau} \int_V {|\nabla (\tau \psi+{\psi}^{\tau}) |}^2-\frac{\lambda}{4\tau} \int_V  [e^{F(u_0+v^{\tau})}-1]^4-[e^{F(u_0+v^{*})}-1]^4  \\ \nonumber
&& \ \ +4\tau {e^{F(u_0+v^{*})}[e^{F(u_0+v^{*})}-1]^2}\psi+4 {e^{F(u_0+v^{*})}[e^{F(u_0+v^{*})}-1]^2}{\psi}^{\tau}  \\ \nonumber
&& \ \  +\frac{\lambda}{\tau} \int_V {e^{F(u_0+v^{*})}[e^{F(u_0+v^{*})}-1]^2}{\psi}^{\tau}-\frac{4\pi M}{|V|\tau}\int_V {\psi}^{\tau}.   \nonumber
\end{eqnarray*}
By Taylor expansion, one has
\begin{eqnarray}\label{k4}
&&\!\!\!\!\!\!\!\!\!-\frac{\lambda}{4\tau}\int_V [e^{F(u_0+v^{\tau})}-1]^4 -[e^{F(u_0+v^{*})}-1]^4+4\tau {e^{F(u_0+v^{*})}[e^{F(u_0+v^{*})}-1]^2}\psi \\ \nonumber
&& \ \ \ \ \ +4{e^{F(u_0+v^{*})}[e^{F(u_0+v^{*})}-1]^2}{\psi}^{\tau} =O(\tau).
\end{eqnarray}
It follows from the fact that $\frac{1}{2\tau} \int_V{|\nabla (\tau \psi+{\psi}^{\tau}) |}^2=O(\tau)$ and (\ref{k4}) that
\begin{eqnarray*}
&& \int_V \nabla v^{*}\cdot \nabla  \psi -\lambda\int_V {e^{F(u_0+v^{*})}[e^{F(u_0+v^{*})}-1]^2}\psi+\frac{4\pi M}{|V|}\int_V \psi\\ \nonumber
&&\geq O(\tau)-\frac{1}{\tau}\int_V (\nabla v^*-\nabla \underline{v})\cdot \nabla  {\psi}^{\tau} -\frac{1}{\tau}\int_V \nabla \underline{v}\cdot \nabla  {\psi}^{\tau} +\frac{\lambda}{\tau}\int_V {e^{F(u_0+\underline{v})}[e^{F(u_0+\underline{v})}-1]^2}{\psi}^{\tau} \\ \nonumber
&&\ \ \  \ \ \ \ \ \ \ -\frac{4\pi M}{|V|\tau}\int_V {\psi}^{\tau} +\frac{\lambda}{\tau}\int_V \left\{{e^{F(u_0+v^{*})}[e^{F(u_0+v^{*})}-1]^2}-{e^{F(u_0+\underline{v})}[e^{F(u_0+\underline{v})}-1]^2}\right\}{\psi}^{\tau}\\ \nonumber
&&\geq O(\tau)-\frac{1}{\tau}\int_V (\nabla v^*-\nabla \underline{v})\cdot \nabla  {\psi}^{\tau}+\frac{\lambda}{\tau}\int_V \left\{{e^{F(u_0+v^{*})}[e^{F(u_0+v^{*})}-1]^2}-{e^{F(u_0+\underline{v})}[e^{F(u_0+\underline{v})}-1]^2}\right\}{\psi}^{\tau}, \\
&& := O(\tau)+I_1 +I_2,
\end{eqnarray*}
where the last inequality follows by the fact that $\underline{v}$ is the subsolution of the equation \eqref{one}.

Denote
\begin{eqnarray*}
V_0=\{x\in V |v^{*}(x)=\underline{v}(x)\} \ \ \ \mbox{and} \ \ \ V_1=\{x\in V |v^*(x)>\underline{v}(x) \}.
\end{eqnarray*}
One can see that ${\psi}^{\tau}(x)=0$ for sufficiently small $\tau$ when $x\in V_1$. Thus, one has
\begin{eqnarray}
I_1&=& -\frac{1}{\tau}\int_V (\nabla v^*-\nabla \underline{v})\cdot \nabla  {\psi}^{\tau} \\ \nonumber
&=&-\frac{1}{\tau }\sum_{x \in V} \sum_{y\sim{x}}
    	\omega_{xy} {[( v^*- \underline{v})(y)-( v^*- \underline{v})(x)]}( {\psi}^{\tau}(y)- {\psi}^{\tau}(x))\\ \nonumber
&=&-\frac{1}{\tau }\sum_{x \in V_0}\left\{ [ \sum_{y\sim{x}, y\in V_0}+\sum_{y\sim{x}, y\in V_1}]
    	\omega_{xy} {[( v^*- \underline{v})(y)-( v^*- \underline{v})(x)]}( {\psi}^{\tau}(y)- {\psi}^{\tau}(x))\right\}\\ \nonumber
   && \ \ \ \ -\frac{1}{\tau }\sum_{x \in V_1}\left\{ [ \sum_{y\sim{x}, y\in V_0}+\sum_{y\sim{x}, y\in V_1}]
    	\omega_{xy} {[( v^*- \underline{v})(y)-( v^*- \underline{v})(x)]}( {\psi}^{\tau}(y)- {\psi}^{\tau}(x))\right\}\\ \nonumber
    &=&-\frac{1}{\tau }\sum_{x \in V_0}\sum_{y\sim{x}, y\in V_1}
    	\omega_{xy} {[( v^*- \underline{v})(y)-( v^*- \underline{v})(x)]}( {\psi}^{\tau}(y)- {\psi}^{\tau}(x))\\ \nonumber
   && \ \ \ \ -\frac{1}{\tau }\sum_{x \in V_1}\sum_{y\sim{x}, y\in V_0}
    	\omega_{xy} {[( v^*- \underline{v})(y)-( v^*- \underline{v})(x)]}( {\psi}^{\tau}(y)- {\psi}^{\tau}(x))\\ \nonumber
&\geq& 0
\end{eqnarray}
for sufficiently small $\tau$.

Similarly, one can obtain that
\begin{eqnarray}
I_2=\frac{\lambda}{\tau}\int_V \left\{{e^{F(u_0+v^{*})}[e^{F(u_0+v^{*})}-1]^2}-{e^{F(u_0+\underline{v})}[e^{F(u_0+\underline{v})}-1]^2}\right\}{\psi}^{\tau}=0.
\end{eqnarray}
Therefore,
\begin{eqnarray}
\int_V \nabla v^{*}\cdot \nabla  \psi -\lambda\int_V {e^{F(u_0+v^{*})}[e^{F(u_0+v^{*})}-1]^2}\psi+\frac{4\pi M}{|V|}\int_V \psi\geq 0\ \  \mbox{as} \ \tau \to 0^+.
\end{eqnarray}
This implies that $<J'(v^*), \psi> \ \geq 0$ for any $\psi \in W^{1,2}(V)$. Replacing $\psi$ with $-\psi$ we obtain the reverse inequality, i.e. $<J'(v^*), \psi> \ \leq 0$. Hence, we obtain that $v^*$ is a critical point for $J$ in $W^{1,2}(V)$.
\end{proof}

  \begin{lem}
    	If $ \lambda>0 $ is sufficiently large, then (\ref{one}) has a solution on $ G $.
    \end{lem}

    \begin{proof}
    	Choose $ {\underline{u}}=-c $ to be a constant function.
    	Then
    	\begin{equation}
    		\Delta \underline{u} =0.
    	\end{equation}
        If $ \lambda $ is sufficient large, one has
    	\begin{equation}
    		{- \lambda}{{e^{F(\underline{u})}} {[{e^{F(\underline{u})}}-1]^2} }
    		+4 \pi \sum_{j=1}^{M} {\delta_{p_j}}
    		\le 0.
    	\end{equation}
    	Therefore,
    	\begin{equation}
    		\Delta \underline{u} \ge {- \lambda}{{e^{F(\underline{u})}} {[{e^{F(\underline{u})}}-1]^2} }
    		+4 \pi \sum_{j=1}^{M} {\delta_{p_j}}.
    	\end{equation}
      This implies that $ \underline{u} $ is a subsolution of (\ref{one}). It follows from Lemma \ref{im} that the equation \eqref{one} has a solution on $G$ for $\lambda$ sufficiently large.

   This completes the proof of the lemma.
    \end{proof}

\begin{lem}
	\label{lem5.1}
	There is a critical value $ \lambda_c $ depending on $ G $ satisfying
    	\begin{equation}
    		\lambda_c\ge \frac{27\pi M}{|V|}
    	\end{equation}
    such that when $ \lambda > \lambda_c $, the equation (\ref{one}) has a solution $ u_\lambda $.
\end{lem}

\begin{proof}
	Denote
	\begin{equation}
		\Lambda =\{ \lambda >0 | \lambda \  \mbox{is  such  that  (\ref{one})  has  a  solution} \}.
	\end{equation}
 We show that $ \Lambda $ is an interval.
    If $ \hat{\lambda} \in \Lambda $, denote by $ \hat{u}=u_0+\hat{v} $ a solution of (\ref{one}) at $ \lambda = \hat{\lambda}$, where $\hat{v}$ is the corresponding solution of the equation \eqref{one2}. If $ \lambda \ge \hat{\lambda}, $ one has
    \begin{align}
    	\Delta {\hat{v}}
    	& = {-\hat{\lambda}}
    	 {e^{F({u_0+\hat{v}})}}{[{e^{F({u_0+\hat{v}})}}-1]^2}
    	 +\frac{4 \pi M}{|V|} \\ \nonumber
    	& \ge {-\lambda}
    	{e^{F({u_0+\hat{v}})}}{[{e^{F({u_0+\hat{v}})}}-1]^2}
    	+\frac{4 \pi M}{|V|}. \nonumber
    \end{align}
    This implies that $ {\hat{v}} $ is a subsolution of the equation (\ref{one2}) on $ G $  for any  $ \lambda \geq \hat{\lambda} $. Then by Lemma \ref{im}, one can obtain the existence of the solutions $v_{\lambda}$ to the equation (\ref{one2}) for any  $ \lambda \geq \hat{\lambda} $. Thus, the equation \eqref{one} has a solution for any  $ \lambda \geq \hat{\lambda} $. It implies that
    \begin{equation*}
    	[\hat{\lambda} ,+ \infty) \subset \Lambda.
    \end{equation*}

Set
\begin{equation}
	{\lambda}_c = \mbox{inf}\{ \lambda | \lambda \in \Lambda \}.
\end{equation}
It follows from (\ref{sev}) that
\begin{equation*}
	\lambda \ge \frac{27 \pi M}{|V|} \ \ \mbox{for any }\lambda >\lambda_c.
\end{equation*}
Taking the limit, one can arrive at
\begin{equation*}
	{\lambda}_c \ge \frac{27 \pi M}{|V|}.
\end{equation*} \end{proof}

In the following, we show that the existence of the solutions to the equation (\ref{one}) on $G$ if $ \lambda =\lambda_c . $
\begin{lem}
	\label{lem5}
	If $ \lambda =\lambda_c $, then (\ref{one}) has a solution $ u_\lambda $.
\end{lem}

\begin{proof}
Choose a sequence ${\lambda}_n \to {\lambda}_c$ as $n\to +\infty$. After passing to a subsequence of $\{\lambda_n\}$, we may assume that $u_{\lambda_n}(x)$ converges to a limit point in $[-\infty,0]$ for any $x\in V$, and we denote this limit by $u(x)$. We prove that $u(x)\in(-\infty,0]$ for any $x\in V$. Suppose not. Then we have two cases:\\
\emph{Case (i)}.
\begin{eqnarray}	
 {\lim_{n \to \infty}}{u_{\lambda_n}}(x)=- \infty  \ \mbox{ for any } x \in V.
 \end{eqnarray}
\emph{Case (ii)}.
There exists $ x \in V$ such that ${\lim_{n \to \infty}}{u_{\lambda_n}}(x)$ exists in $(-\infty,0]$.

If \emph{Case (i)} holds, when ${\lambda_n \to \lambda_c}$, summing the two sides of (\ref{one}) yields that the left hand side of (\ref{one}) approaches $0$ and the right hand side remains at least $2\pi$. This leads to a contradiction.

If \emph{Case (ii)} holds, we split $ V $ into two subsets $ V_1 $ and $ V_2 $, where
	\begin{align}\label{V}
		V_1&=\left\{ x\in V\Big| \lim_{n \to \infty}{u_{\lambda_n}}(x) = - \infty \right\}\nonumber\\
        V_2&=\left\{ x\in V\Big| \lim_{n \to \infty}{u_{\lambda_n}}(x) \   \mbox{exists in }(-\infty,0] \right\}.
	\end{align}
	
	If $ V_1 $ is empty, then Lemma \ref{lem5} holds. In the following, we consider that both $ V_1 $ and $ V_2 $ are non-empty sets.
	
	Since $ G $ is a connected finite graph, one may choose $ x_2 \in V_2 $ such that there exists $x_1\in V_1$ satisfying $x_1\sim x_2$. Then
	\begin{align}
		\Delta {u_{\lambda_n}}(x_2)
		&= \frac{1}{ \mu (x_2)} \sum_{y\sim{x_2}} \omega_{y{x_2}}
		[{u_{\lambda_n}} (y) -{u_{\lambda_n}}(x_2) ] \\ \nonumber
		&= \frac{1}{ \mu (x_2)} \sum_{y\sim{x_2}, y \in V_1} \omega_{y{x_2}}
		[{u_{\lambda_n}} (y) -{u_{\lambda_n}}(x_2) ]
		+  \frac{1}{ \mu (x_2)} \sum_{y\sim{x_2}, y \in V_2} \omega_{y{x_2}}
		[{u_{\lambda_n}} (y) -{u_{\lambda_n}}(x_2) ]\\ \nonumber
        &:= I_1+I_2.
	\end{align}
Now we calculate $I_1$ and $I_2$ respectively. It follows from (\ref{V}) that
    \begin{equation}
    	I_1=\frac{1}{ \mu (x_2)}\sum_{y\sim{x_2}, y \in V_1} \omega_{y{x_2}}[{u_{\lambda_n}} (y) -{u_{\lambda_n}}(x_2) ] \to - \infty \ \ \mbox{as} \ n \to \infty.
    \end{equation}
   For $I_2$, since $x_2, y \in V_2$, the limit of $ {u_{\lambda_n}} $ as $ {\lambda_n}$ tends to $ \lambda_c $ exists. One can obtain that
    \begin{equation}
    	\frac{1}{ \mu (x_2)}\sum_{y\sim{x_2}, y \in V_2} \omega_{y{x_2}}[{u_{\lambda_n}} (y) -{u_{\lambda_n}}(x_2) ]
    \end{equation}
    is bounded. Taking the limit $ {\lambda_n} \to \lambda_c $, one has
    \begin{equation}
    	\label{12}
    	\Delta {u_{\lambda_n}}(x_2) \to -\infty.
    \end{equation}
    However, it follows from (\ref{one}) that $\Delta {u_{\lambda_n}}(x_2)$ is finite, which is a contradiction with (\ref{12}). This completes the proof of the lemma.
\end{proof}
Combining Lemma \ref{lem5.1} and Lemma \ref{lem5}, one can obtain that the equation (\ref{one}) has a solution $u_{\lambda}$ on $ G $ when $ \lambda \ge \lambda_c $. Hence, \textbf{the statement (i) of Theorem \ref{thma} is proved}.\\

In the following, with the aid of Lemma \ref{im}, we prove that there exists a solution $\tilde{u}_{\lambda}$ to the equation \eqref{one} satisfying $\tilde{u}_{\lambda}$ is increasing with respect to $\lambda$.
\begin{lem} \label{im1}
	There is a critical value $ \lambda_c $ depending on $ G $ satisfying
    	\begin{equation*}
    		\lambda_c\ge \frac{27\pi M}{|V|}
    	\end{equation*}
    such that equation (\ref{one}) has a solution $\tilde{u}_{\lambda}$ on $ G $ for any $ \lambda \geq \lambda_c $, and
    \begin{eqnarray}
    \tilde{u}_{\lambda_1}\geq\tilde{u}_{\lambda_2} \ \ \mbox{when} \ \lambda_1>\lambda_2.
    \end{eqnarray}
\end{lem}
\begin{proof}
Equation (\ref{one}) has a solution $u_{\lambda}$ on $ G $ when $ \lambda \ge \lambda_c $, and we denote the set of solvability parameters by $\tilde{\Lambda}:=[{\lambda}_c, +\infty)$. We fix a solution of equation~(\ref{one}) with parameter $\lambda=\lambda_c$, and denote it by $u_{\lambda_c}$. For any fixed $m\in \mathbb{N}$, let
\begin{eqnarray*}
\tilde{\Lambda}=\bigcup_{k=0}^{\infty}\tilde{\Lambda}_{km},
\end{eqnarray*}
where
\begin{eqnarray*}
\tilde{\Lambda}_{km}:=[{\lambda}_c+\frac{k}{2^m}, {\lambda}_c+\frac{k+1}{2^m}] \ \ (k\in \mathbb{N}).
\end{eqnarray*}
Note that for any $\lambda>\lambda_c$, $u_{\lambda_c}$ is a subsolution of the equation \eqref{one} with parameter $\lambda$. It then follows from the argument of Lemma \ref{im} that for any $\lambda\in({\lambda}_c, {\lambda}_c+\frac{1}{2^m}]$, there exists a solution $u_{\lambda,m}$ of \eqref{one} such that
\begin{eqnarray*}
u_{\lambda,m}\geq u_{{\lambda}_c},\ \ \ \lambda\in \tilde{\Lambda}_{0m} .
\end{eqnarray*}
Now consider the solution $u_{{\lambda}_c+\frac{1}{2^m},m}$. Similarly, since $u_{{\lambda}_c+\frac{1}{2^m},m}$ is a subsolution of equation~\eqref{one} with parameter $\lambda>\lambda_c+\frac1{2^m}$, by the argument of Lemma~\ref{im}, one can find a solution $u_{\lambda,m}$ of the equation~\eqref{one} for any $\lambda\in ({\lambda}_c+\frac{1}{2^m}, {\lambda}_c+\frac{2}{2^m}]$ such that
\begin{eqnarray*}
u_{\lambda,m}\geq u_{{\lambda}_c+\frac{1}{2^m},m},\ \ \ \lambda \in \tilde{\Lambda}_{1m}.
\end{eqnarray*}
Repeating the process above yields that for any $\lambda\in[{\lambda}_c, +\infty)$, there exists a solution $u_{\lambda,m}$ of equation~\eqref{one} with parameter $\lambda$, and if $\lambda\in \tilde{\Lambda}_{km}$ for some $k\in\{0,1,2,...\}$, then
\begin{eqnarray*}
u_{\lambda,m}\geq u_{{\lambda}_c+\frac{k}{2^m},m}.
\end{eqnarray*}
Here, for convenience, we let $u_{\lambda_c,m}=u_{\lambda_c}$. Then we have constructed an increasing sequence
\begin{eqnarray}\label{k5}
u_{\lambda_c,m}\leq u_{\lambda_c+\frac1{2^m},m}\leq u_{\lambda_c+\frac2{2^m},m}\leq\cdots
\end{eqnarray}
for any $m\in\mathbb N$.

Now define
\begin{eqnarray}
Q:=\left\{{\lambda}_c+\frac{j}{2^m}\Big| m\in \mathbb{N}, \ j\in \mathbb{N} \ \mbox{odd}\right\}.
\end{eqnarray}
We rearrange the elements in the countable set $Q$ and list them as $Q=\{q_l:l\in\mathbb N\}$.

For any $q_l\in Q$, we have constructed a sequence of solutions $\{u_{q_l,m}\}_{m\in\mathbb N}$ of equation~\eqref{one} with parameter $\lambda=q_l$. Since the functions $u_{q_l,m}$ $(l,m\in\mathbb N)$ are defined on the finite connected graph $G=(V,E)$ and
\begin{eqnarray*}
 u_{\lambda_c}\leq u_{q_l,m}\leq0\ \ (\forall l,m\in\mathbb N)
\end{eqnarray*}
there exists a subsequence $\{n_i\}\subset \mathbb N$ such that, for any $q_l\in Q$, $\{u_{q_l,n_i}\}$ converges as $i\to\infty$.

Define
\begin{eqnarray*}
u_{q_l}:=\lim_{i\to\infty}u_{q_l,n_i},
\end{eqnarray*}
and let
\begin{eqnarray}
U=\{u_{q_l}|q_l\in Q\}.
\end{eqnarray}
It follows from (\ref{k5}) that if $q_{l_1},q_{l_2}\in Q$, $q_{l_1}<q_{l_2}$ and $i$ is sufficiently large, then
\begin{eqnarray*}
u_{q_{l_1},n_i}\leq u_{q_{l_2},n_i}
\end{eqnarray*}
and hence by taking limits as $i\to\infty$ we have
\begin{eqnarray}\label{k7}
u_{q_{l_1}}\leq u_{q_{l_2}}\ \ (q_{l_1}<q_{l_2}).
\end{eqnarray}
This implies that $\{u_\lambda:\lambda\in Q\}$ is an increasing sequence with respect to the parameter $\lambda\in Q$.

Now for any $\lambda\geq\lambda_c$, there exists an increasing sequence $\{q_l\}\subset Q$ such that $q_l\to\lambda$. The increasing sequence $\{u_{q_l}\}$ then has a limit and we define
\begin{eqnarray}\label{k6}
 \tilde{u}_{\lambda}:=\lim_{l\to\infty} u_{q_l}.
\end{eqnarray}
Then $\tilde u_\lambda$ is a solution of equation~\eqref{one} for the parameter $\lambda\in\tilde\Lambda$. Let
\begin{eqnarray}
\bar{U}=\{\tilde{u}_{\lambda}
|\lambda\in [{\lambda}_c, +\infty) \}.
\end{eqnarray}
It follows from (\ref{k7}) and (\ref{k6}) that if $\lambda_1>\lambda_2$, then
\begin{eqnarray}
\tilde{u}_{\lambda_1}\geq\tilde{u}_{\lambda_2}.
\end{eqnarray}
Hence, $\tilde{u}_{\lambda}$ is increasing with respect to $\lambda\in[\lambda_c,\infty)$. This completes the proof of the lemma.
\end{proof}
By Lemma \ref{im1}, \textbf{the statement (ii) of Theorem \ref{thma} is proved}.

From Lemma \ref{im1} and the maximum principle, we can obtain that $\tilde{u}_{\lambda}>\tilde{u}_{{\lambda}_c}$ when $\lambda>\lambda_c$. Thus, $\tilde{v}_{\lambda}(=\tilde{u}_{\lambda}-u_0)$ is the local minimum point of the functional $J$ for any $\lambda >\lambda_c$. If $v_{\lambda}$ obtained by Lemma \ref{lem5.1} is different from $\tilde{v}_{\lambda}$. Then we have already found a second solution. If $v_{\lambda}$ is equal to $\tilde{v}_{\lambda}$ as the local minimum of the functional $J$, there exists ${\rho}_{\lambda}$ such that
\begin{eqnarray}
J(v_{\lambda})\leq J(v) \ \ \mbox{for any} \ \ \|v-v_{\lambda}\|\leq {\rho}_{\lambda}.
\end{eqnarray}

Then we have the following two cases.

\textbf{Case (i)}.
For any $\rho\in(0, {\rho}_{\lambda})$, one has
\begin{eqnarray}
\inf_{\|v-v_{\lambda}\|=\rho} J=J(v_{\lambda})(:={\eta}_{\lambda}).
\end{eqnarray}
Thus, there is a local minimum $v_{\rho}$ of $J$ satisfying $\|v_{\rho}-v_{\lambda}\|=\rho$ and $J(v_{\lambda})={\eta}_{\lambda}$ for any $\rho\in(0, {\rho}_{\lambda})$. Therefore, one can obtain a one-parameter family of solutions to the equation \eqref{one}.

\textbf{Case (ii)}.
There exists $ r_{\lambda}\in(0, {\rho}_{\lambda})$ such that
\begin{equation}\label{g1}
    	J(v_\lambda)< \inf_{\|w-v_\lambda\|=r_{\lambda}}J(w).
    \end{equation}

In the following, we prove that there is a second solution of (\ref{one}) on $ G $ for $ \lambda> \lambda_c $ by Mountain Pass theorem. We first show that $ J(v) $ satisfies the Palais-Smale condition in order to apply Mountain Pass theorem for the functional $J$.

\begin{lem}\label{lemm}
	Every sequence $ \{v_n\} \subset W^{1,2}(V) $ admits a convergent subsequence if $ \{v_n\} \subset W^{1,2}(V) $ satisfies
	\begin{flalign}
		\label{psa}
		 &(1)\  J(v_n)  \to c \ \mbox{for some constant} \ c, \  as \  n \to + \infty, &&\\
		\label{psb}
		 &(2)\  \Vert J'(v_n) \Vert \to 0 \ \mbox{in} \ W^{1,2}(V), \  as \  n \to + \infty.&&
	\end{flalign}
\end{lem}

\begin{proof}
	It follows from (\ref{psb}) that
	\begin{equation}
		\label{ps1}
		| \int_V \nabla{v_n} \cdot \nabla \varphi
		 -\lambda \int_V e^{F(u_0+v_n)} [e^{F(u_0+v_n)}-1]^2 \varphi
		 +\frac{4\pi M}{|V|}\int_V \varphi|
		\le \varepsilon_n \| \varphi \|,
	\end{equation}
    where $\varphi \in W^{1,2}(V) $ and $ \varepsilon_n \to 0 $ as $ n \to +\infty $.
    By taking $ \varphi = -1  $, one has
    \begin{equation}
    	\label{ten}
    	0<c \le \int_V e^{F(u_0+v_n)} [e^{F(u_0+v_n)}-1]^2 \le C,
    \end{equation}
    where $ c $ and $ C $ are constants depending only on $\lambda$.

   We claim that $ \{v_{n} (x)  \} $ is bounded for any $  x \in V $. Suppose that the claim is not true, there exists a point $ x^k \in V$ and a subsequence $ \{v_{n_k} (x^k)  \}\subset \{v_{n}   \} $ tends to $ +\infty $ or $ -\infty $ as $k \to +\infty$. Without loss of generality, we assume that $ \{v_{n_k} (x^k)  \}$ tends to $ +\infty $ as $k \to +\infty$.

   In the following, for the subsequence $\{v_{n_k}\}$, we first show that there exists a point $ \hat{x} \in V $ such that $ \{v_{n_k}(\hat{x})\} $ is bounded. If not, one can obtain that $ \{v_{n_k} (y)  \} $ tends to $ +\infty $ or $ -\infty $ for any point $ y \in V $ as $k \to +\infty$. Therefore, one can derive that
    \begin{equation}
    	\int_V e^{F((u_0+v_{n_k}) (y)  )} [e^{F({(u_0+v_{n_k})(y)  })}-1 ]^2\rightarrow 0, \ \ \mbox{as} \ k\rightarrow +\infty,
    \end{equation}
    which contradicts with (\ref{ten}).

    Let
    \begin{equation}
    	L_k := \max\limits_{V} |v_{n_k}(x) |=v_{n_k}(x^k).
    \end{equation}
    From (\ref{psa}), one has
    \begin{align}
    	\label{2}
    	J(v_{n_k}) &= \frac{1}{2} \int_V {|\nabla v_{n_k} |}^2 +\frac{\lambda}{4} \int_V [e^{F(u_0+v_{n_k})}-1]^4 +\frac{4\pi M}{|V|}\int_V v_{n_k} \nonumber \\
    	&=c + o_n(1).
    \end{align}
   Combining (\ref{gra}) and (\ref{2}) yields that
    \begin{equation}
    	c \ge \frac{1}{2} \sum_{x \in V} \sum_{y\sim{x}}
    	\omega_{xy}	{[ v_{n_k}(y) -v_{n_k}(x)  ]}^2
    	- 4\pi M L_k.
    \end{equation}
   It implies that
   \begin{equation}\label{l1}
   	| v_{n_k}(x) - v_{n_k}(y)   |
   \le C \sqrt{ c +4\pi M L_k},\  x \sim y\ ,\mbox{for any }x, \ y\in V,
   \end{equation}
    where $ C $ is a constant.

    Since the graph $G$ is finite and connected, there is a path $\{x_i\}_{i=0}^{l}\subset G$ between $\hat{x}$ and $x^k$ for $l\in \mathbb{N}$, that is
    \begin{eqnarray}\label{ll}
    \hat{x}=x_0\sim x_1 \sim x_2 \sim \cdots \sim x_l=x^k.
    \end{eqnarray}
    Combining (\ref{l1}) and (\ref{ll}), one can obtain that
    \begin{align}
    	v_{n_k}(\hat{x}) & \ge v_{n_k}(x^k) - (l-1)C\sqrt{c +4\pi M L_k} \nonumber \\
    	             & = L_k -(l-1)C \sqrt{c +4\pi M L_k} \nonumber \\
    	             & \to + \infty, \ \ \mbox{as} \ k\rightarrow +\infty.
    \end{align}
    This is a contradiction with the conclusion that $ \{v_{n_k}(\hat{x})\} $ is bounded. Therefore, $ \{v_{n} (x)  \} $ is bounded for any $  x \in V $. This implies that $ \{v_{n} (x)  \} $ is bounded in $W^{1,2}(V)$. Applying Sobolev embedding theorem(Lemma \ref{le1}), there exists a subsequence ${v_{n_k} (x)  } $ converges to some function $v(x)$ in $W^{1,2}(V)$. Hence, the Palais-Smale condition holds.
\end{proof}
For $Q>0$, one has
    \begin{eqnarray*}
    J(v_{\lambda}-Q)-J(v_{\lambda})&=& \frac{\lambda}{4} \int_V [e^{F(u_0+v_{\lambda}-Q)}-1]^4 - \frac{\lambda}{4} \int_V [e^{F(u_0+v_{\lambda})}-1]^4 -4 \pi MQ\\ \nonumber
                                   &\to& -\infty, \ \ \mbox{as} \ Q \to +\infty.
    \end{eqnarray*}
    Thus, one has
    \begin{eqnarray}
    J(v_{\lambda}-Q_0)<J(v_{\lambda})\ \ \mbox{for some} \ Q_0>r_{\lambda}>0 \ \mbox{sufficiently large}.
    \end{eqnarray}

Let
\begin{eqnarray*}
\Gamma:=\{\gamma\in C([0,1], W^{1,2}(V))|\gamma{(0)}= v_{\lambda}, \ \gamma{(1)}=v_{\lambda}-Q_0\},
\end{eqnarray*}
and set
\begin{eqnarray*}
c=\inf_{\gamma\in \Gamma}\max_{0\leq t \leq 1}J(\gamma(t)).
\end{eqnarray*}

It follows from (\ref{g1}) that $c>J(v_{\lambda})$. From Lemma \ref{lemm}, applying Mountain Pass theorem(Lemma \ref{le2}), we conclude that $c$ defines a critical value of $J$. Since $c>J(v_{\lambda})$, the corresponding critical point yields to a second solution for \eqref{one2}. Therefore, there exists a second solution to the equation \eqref{one}.  Hence, \textbf{the statement (iii) of Theorem \ref{thma} is proved}.

This completes the proof of Theorem \ref{thma}.

\section{The Proof of Theorem \ref{thmc}}
In this section, we prove Theorem \ref{thmc}. From the analysis in Section 3, we know that $\tilde{v}_{\lambda}$ is a local minimum of the functional $J(v)$ for $\lambda > \lambda_c$. In the following, we show that $\tilde{v}_{\lambda}$  is a strict local minimum of the functional $J$ for almost everywhere $\lambda > \lambda_c$.
\begin{proof}
	It follows from Lemma \ref{im1} that $\tilde{v}_{\lambda}$  is increasing with respect to $\lambda$ for $ \lambda> \lambda_c $. This guaranties that
    \begin{equation}
    	\label{mea}
    	\bar{v}_\lambda :=  \frac{d\tilde{v}_{\lambda}}{d \lambda}
    \end{equation}
   is well-defined for almost everywhere  $ \lambda>\lambda_c $, and therefore $ \bar{v}_\lambda \ge 0 \ a.e.$ in $V$.

    Taking the derivative of both sides of the equation \eqref{one2} with respect to $\lambda$ yields that
    \begin{equation}
    	\label{star}
    	\Delta \bar{v}_\lambda= -e^{F(u_0+\tilde{v}_\lambda)}  [e^{F(u_0+\tilde{v}_\lambda)} -1 ]^2
    	-\lambda e^{F(u_0+\tilde{v}_\lambda)} [1-3e^{F(u_0+\tilde{v}_\lambda)}] \cdot \bar{v}_\lambda.
    \end{equation}
    From the fact that $\tilde{v}_\lambda$ is the critical point of $ J(v) $, one can derive that
    \begin{equation}
    	0=<J'(\tilde{v}_\lambda), \varphi> =\int_V \nabla \tilde{v}_\lambda \cdot \nabla \varphi -\lambda \int_V
    	 e^{F(u_0+\tilde{v}_\lambda)}  [e^{F(u_0+\tilde{v}_\lambda)} -1 ]^2 \varphi +\frac{4\pi M}{|V|} \int_V\varphi,
    \end{equation}
    and
	\begin{flalign}\label{e}
		<J''(\tilde{v}_\lambda)\psi, \varphi>
		&=\int_V \nabla \psi \cdot \nabla \varphi -\lambda \int_V
		e^{F(u_0+\tilde{v}_\lambda)}  [1-3e^{F(u_0+\tilde{v}_\lambda)} ] \varphi \psi
	\end{flalign}
for any $\varphi, \psi\in W^{1,2}(V)$.

From (\ref{e}), one can see that
	\begin{equation}
		J''(\tilde{v}_\lambda)h= -\Delta h - \lambda  e^{F(u_0+\tilde{v}_\lambda)}  [1-3e^{F(u_0+\tilde{v}_\lambda)} ]h \ \ \mbox{for any } h \in W^{1,2}(V).
	\end{equation}
    Then
    \begin{equation*}
    	<J''(\tilde{v}_\lambda)h,g>=<J''(\tilde{v}_\lambda)g,h>\ \mbox{for any } h,g \in W^{1,2}(V).
    \end{equation*}
    This implies that $ J''(\tilde{v}_\lambda) $ is a symmetric and linear operator from $W^{1,2}(V)$ to $W^{1,2}(V)$.

    Denote $\mu_1, \mu_2, \cdots, \mu_{|V|}$ be the eigenvalues of $ J''(\tilde{v}_\lambda) $ satisfying $\mu_1\leq \mu_2 \leq \cdots \leq \mu_{|V|}$ and $ {\theta}_1,\dots, {\theta}_{|V|} $ be the related eigenvectors of $ J''(\tilde{v}_\lambda) $, i.e.
    \begin{equation*}
    	J''(\tilde{v}_\lambda) {\theta}_i=\mu_i {\theta}_i,
    \end{equation*}
    where
    $<{\theta}_i, {\theta}_j>=\delta_{ij},$ for any $i,j=1,2, \cdots, |V|$.

    Then the eigenvalue
    \begin{equation*}
    	\mu_1= \inf_{h \in {W^{1,2}(V)}} \frac{<J''(\tilde{v}_\lambda)h,h>}{<h,h>}.
    \end{equation*}
    In the following, we prove that $\mu_1>0$.

Noting that
\begin{flalign}\label{f}
		<J''(\tilde{v}_\lambda)h, h>
		&=\int_V |\nabla h|^2 -\lambda \int_V
		e^{F(u_0+\tilde{v}_\lambda)}  [1-3e^{F(u_0+\tilde{v}_\lambda)} ] h^2\\
        &= \sum_{x \in V} \sum_{y\sim{x}}
    	\omega_{xy} {[h(y)-h(x)]}^2-\lambda \int_V
		e^{F(u_0+\tilde{v}_\lambda)}  [1-3e^{F(u_0+\tilde{v}_\lambda)} ] h^2.\nonumber
	\end{flalign}
Let $ \tilde{h}=(|h_1|,|h_2|,\dots,|h_{|V|}|)$ for any $h=(h_1,h_2,\dots,h_{|V|}) \in W^{1,2}(V)$, it follows from (\ref{f}) that
\begin{equation}
	<J''(\tilde{v}_\lambda)\tilde{h},\tilde{h}>\ \leq \ <J''(\tilde{v}_\lambda)h,h>.
\end{equation}
Thus, one can obtain that all of the components of $ \theta_1 $ have the same sign. Without loss of generality, we assume that they are all non-negative, i.e.
\begin{eqnarray}
\theta_{1,x_i}\geq 0, \ \ \mbox{for} \ \ 1\leq i\leq |V|, \ x_i\in V.
\end{eqnarray}

Taking $\psi=\bar{v}_{\lambda}$, it follows from (\ref{e}) that
\begin{eqnarray}\label{4}
<J''(\tilde{v}_\lambda)\bar{v}_\lambda, \varphi>=\int_V e^{F(u_0+\tilde{v}_\lambda)}[e^{F(u_0+\tilde{v}_\lambda)} -1 ]^2 \varphi
\end{eqnarray}
for any $\varphi\in W^{1,2}(V)$. Let $\varphi={\theta}_1$, one has
\begin{equation}
	<J''(u_\lambda)\bar{v}_\lambda, {\theta}_1>=\int_V e^{F(u_0+\tilde{v}_\lambda)}[e^{F(u_0+\tilde{v}_\lambda)} -1 ]^2 {\theta_1} >0.
\end{equation}
Since $\bar{v}_\lambda\in W^{1,2}(V)$, then $\bar{v}_\lambda$ can be represented by
\begin{equation*}
	\bar{v}_\lambda=\sum_{i=1}^{|V|} a_i {\theta_i},
\end{equation*}
where $a_i=<\bar{v}_\lambda, \theta_i>$, $1\leq i \leq |V|$. One can see that $a_1\geq 0$, and
 \begin{equation*}
 	0< \ <J''(\tilde{v}_\lambda)\bar{v}_\lambda, {\theta_1}>=a_1 \mu_1.
 \end{equation*}

Thus, one has
\begin{eqnarray}
\mu_1 >0.
\end{eqnarray}
Therefore, $ \tilde{v}_\lambda $ is the strictly minimum of $ J(\tilde{v}_\lambda) $  for almost everywhere $\lambda> \lambda_c$. This implies that for almost everywhere $\lambda$,  there exists $ r_{\lambda}$ such that

    \begin{equation}\label{g}
    	J(\tilde{v}_\lambda)< \inf_{\|w-\tilde{v}_\lambda\|=r_{\lambda}}J(w).
    \end{equation}

 This completes the proof of Theorem \ref{thmc}.

\end{proof}

\section{Acknowledgement}
The authors would like to thank Professor Genggeng Huang for his interest in this work and for many helpful discussions.

\end{document}